\documentclass[11pt,twoside,letterpaper]{article}
\usepackage{pslatex}
\usepackage{fancyhdr}
\usepackage{graphicx}
\usepackage{geometry}

\def\figurename{} 
\makeatletter
\renewcommand{\fnum@figure}[1]{\figurename~\thefigure.}
\makeatother

\def\tablename{Table} 
\makeatletter
\renewcommand{\fnum@table}[1]{\centering\bfseries{\tablename~\thetable.}}
\makeatother

\usepackage{amsmath,amsthm,amsfonts}
\newtheorem{thm}{Theorem}[section]

\newtheorem{cor}[thm]{Corollary}

\newtheorem{lem}[thm]{Lemma}

\newtheorem{ex}[thm]{Example}

\theoremstyle{definition}
\newtheorem{defn}[thm]{Definition}

\theoremstyle{remark}
\newtheorem{rem}[thm]{Remark}
\numberwithin{equation}{section}


\begin{document}
\title{\bfseries\scshape{On a linear partial differential equation of the higher
order in two variables with initial condition whose coefficients
are  real-valued simple step functions}}
\author{\bfseries\itshape Gogi Pantsulaia\thanks{E-mail address: g.pantsulaia@gtu.ge},
Givi Giorgadze\thanks{E-mail address: g.giorgadze@gtu.ge}
\\
I. Vekua Institute of Applied Mathematics, Tbilisi State
University,  B. P.
0143,\\ University St. 2, Tbilisi, Georgia \\
Department of Mathematics,~ Georgian Technical University, B. P.
0175,  \\ Kostava St. 77, Tbilisi 75, Georgia }

\date{}
\maketitle \thispagestyle{empty} \setcounter{page}{1}
\thispagestyle{fancy} \fancyhead{}
\fancyhead[L]{} \fancyfoot{}
\renewcommand{\headrulewidth}{0pt}

\begin{abstract}By using the method developed in the paper [{\it Georg. Inter. J. Sci. Tech.,} Volume 3,  Issue 1 (2011), 107-129], it is obtained
a representation  in  an explicit form of the weak solution of a
linear partial differential equation of the higher order in two
variables with initial condition whose coefficients are
real-valued simple step functions
\end{abstract}

%
%
%

\noindent \textbf{2000 Mathematics Subject Classification:
}Primary 34Axx ; Secondary 34A35, 34K06.

\vspace{.08in} \noindent \textbf{Key words and phrases:}
 linear partial differential equation of the higher order in two variables, Fourier differential operator.

\section{Introduction}

In \cite{Pan-Gio11} has been obtained a representation  in  an
explicit form of the solution of the linear partial differential
equation of the higher order in two variables with initial
condition whose coefficients were real-valued coefficients. The
aim of the present manuscript is resolve an analogous problem for
a linear partial differential equation of the higher order in two
variables with initial condition whose coefficients are
real-valued simple step functions.

The paper is organized as follows.

In Section 2, we consider some auxiliary results obtained in the
paper \cite{Pan-Gio11}.  In Section 3, it is obtained  a
representation in an explicit form of the weak solution of the
partial differential equation of the higher order in two variables
with initial condition whose coefficients are  real-valued simple
step functions.

\section{Some auxiliary results}

\begin{defn}
Fourier   differential operator
$(\mathcal{F}){\frac{\partial}{\partial x}}$ in $R^{\infty}$ is
defined as follows :

$$(\mathcal{F}){\frac{\partial}{\partial x}}
\begin{pmatrix}
  \frac{a_0}{2} \\
   a_1\\
   b_1\\
   a_2\\
   b_2\\
   a_3\\
   b_3\\
  \vdots
\end{pmatrix}=\begin{pmatrix}
   0&  0&  0& 0 & 0 &  0&  0& \dots&  \\
   0&  0&  \frac{1\pi}{l}&  0&  0&  0&  0& \dots&  \\
   0&  -\frac{1\pi}{l}&  0&  0&  0&  0&  0& \dots&  \\
   0&  0&  0&  0&  \frac{2\pi}{l}& 0 &  0& \dots&  \\
   0& 0 &  0&  -\frac{2\pi}{l}&  0&  0&  0&  \dots& \\
   0&  0&  0&  0&  0&  0&  \frac{3\pi}{l}& \ddots&  \\
   0&  0&  0&  0&  0&  -\frac{3\pi}{l}&  0& \ddots&  \\
   \vdots&  \vdots&  \vdots&  \vdots&  \vdots&  \vdots&  \ddots& \ddots&
\end{pmatrix}   \times \begin{pmatrix}
  \frac{a_0}{2} \\
   a_1\\
   b_1\\
   a_2\\
   b_2\\
   a_3\\
   b_2\\
  \vdots
\end{pmatrix}. \eqno(2.1)
$$
\end{defn}
For $n \in \mathbb{N} $, let $FD^{n}[-l,l[$  be  a vector space of all
$n$-times differentiable functions   on $[-l,l[$ such that for
arbitrary $0 \le k \le n-1$, a series obtained by a
differentiation term by term of the Fourier series of $f^{(k)}$
pointwise converges to $f^{(k+1)}$ for all $x \in [-l,l[$.

\begin{lem}
Let $f \in FD^{(1)}[-l,l[$. Let $G_M$ be an embedding of the
$FD^{(1)}[-l,l[$ in to $R^{\infty}$ which sends a function to a
sequence of real numbers consisting from its Fourier coefficients.
i.e., if
$$f(x)= \frac{c_0}{2}+\sum_{k=1}^{\infty}c_k\cos(\frac{k\pi x}{l})+ d_k\sin(\frac{k\pi x}{l})~(x \in [-l,l[),$$
 then $G_F(f)=(\frac{c_0}{2},c_1,d_1, c_2,d_2,\dots)$.
 Then, for $f \in FD^{(1)}[-l,l[$, the following equality
 $$\Big(G_F^{-1} \circ(\mathcal{F}){\frac{\partial}{\partial x}}\circ G_F \Big)(f)=
 \frac{\partial}{\partial x}(f)  \eqno(2.2)$$
 holds.
\end{lem}

\begin{proof} Assume that for $f \in FD^{(1)}[-l,l[$, we have the following representation
$$
f(x)= \frac{c_0}{2}+\sum_{k=1}^{\infty}c_k \cos\Big(\frac{k \pi
x}{l}\Big)+d_k \sin \Big(\frac{k \pi x}{l}\Big)~(x \in [-l,l[).$$
By the definition of the class $FD^{(1)}[-l,l[$, we have
$$
\frac{d}{d x}(f)=\frac{\partial}{\partial
x}(\frac{c_0}{2}+\sum_{k=1}^{\infty}c_k \cos\Big(\frac{k \pi
x}{l}\Big)+d_k \sin \Big(\frac{k \pi x}{l}\Big))=
$$
$$
\sum_{k=1}^{\infty}c_k \frac{\partial}{\partial
x}(\cos\Big(\frac{k \pi x}{l}\Big))+d_k \frac{\partial}{\partial
x}(\sin \Big(\frac{k \pi x}{l}\Big))=
$$
$$
\sum_{k=1}^{\infty}-c_k \frac{k \pi}{l}\sin\Big(\frac{k \pi
x}{l}\Big)+d_k \frac{k \pi}{l}\cos \Big(\frac{k \pi x}{l}\Big)=
$$
$$\sum_{k=1}^{\infty}\frac{k \pi d_k}{l}\cos \Big(\frac{k \pi x}{l}\Big)-
\frac{k \pi c_k}{l}\sin\Big(\frac{k \pi x}{l}\Big).
$$
By the definition of the composition of mappings, we have
$$
\Big(G_F^{-1} \circ(\mathcal{F}){\frac{\partial}{\partial x}}\circ
G_F \Big)(f)= G_F^{-1}((\mathcal{F}){\frac{\partial}{\partial
x}}((G_F (f))))= G_F^{-1}((\mathcal{F}){\frac{\partial}{\partial
x}}
\begin{pmatrix}
  \frac{c_0}{2} \\
   c_1\\
   d_1\\
   c_2\\
   d_2\\
   c_3\\
   d_3\\
  \vdots
\end{pmatrix})=
$$
$$
G_F^{-1}(\begin{pmatrix}
   0&  0&  0& 0 & 0 &  0&  0& \dots&  \\
   0&  0&  \frac{1\pi}{l}&  0&  0&  0&  0& \dots&  \\
   0&  -\frac{1\pi}{l}&  0&  0&  0&  0&  0& \dots&  \\
   0&  0&  0&  0&  \frac{2\pi}{l}& 0 &  0& \dots&  \\
   0& 0 &  0&  -\frac{2\pi}{l}&  0&  0&  0&  \dots& \\
   0&  0&  0&  0&  0&  0&  \frac{3\pi}{l}& \ddots&  \\
   0&  0&  0&  0&  0&  -\frac{3\pi}{l}&  0& \ddots&  \\
   \vdots&  \vdots&  \vdots&  \vdots&  \vdots&  \vdots&  \ddots& \ddots&
\end{pmatrix}   \times \begin{pmatrix}
  \frac{c_0}{2} \\
   c_1\\
   d_1\\
   c_2\\
   d_2\\
   c_3\\
   d_3\\
  \vdots
\end{pmatrix})=
$$
$$
G_F^{-1}(\begin{pmatrix}
  0 \\
   \frac{1 \pi d_1}{l} \\
   -\frac{1 \pi c_1}{l}\\
   \frac{2 \pi d_2}{l}\\
  -\frac{2 \pi c_2}{l}\\
   \frac{3 \pi d_3}{l}\\
  -\frac{3 \pi c_3}{l}\\
  \vdots
\end{pmatrix})=
\sum_{k=1}^{\infty}\frac{k \pi d_k}{l}\cos \Big(\frac{k \pi
x}{l}\Big)- \frac{k \pi c_k}{l}\sin\Big(\frac{k \pi x}{l}\Big).
$$
\end{proof}

By the scheme used in the proof of Lemma 2.1, we can get the
validity of the following assertion.

\begin{lem} Let
$G_M$ be an embedding of the $FD^{n}[-l,l[$ in to $R^{\infty}$
which sends a function to a sequence of real numbers consisting
from its Fourier coefficients.

 Then, for $f \in FD^{(n)}[-l,l[$ and $A_k \in R (0 \le k \le n)$, the following equality
 $$\Big(G_F^{-1} \circ \Big( \sum_{k=0}^nA_k((\mathcal{F}){\frac{\partial}{\partial x}})^k \Big) \circ G_F \Big)(f)=
 \sum_{k=0}^na_k\frac{\partial^k}{\partial x^k}(f)  \eqno(2.3)$$
 holds.
\end{lem}

\begin{ex} \cite{Gant66} If $A$ is the real matrix
$$\begin{pmatrix}
  \sigma& \omega&  \\
  -\omega& \sigma&
\end{pmatrix}, \eqno (2.4)
$$
then
$$
e^{tA}=e^{\sigma t}\begin{pmatrix}
  \cos (\omega t)& \sin (\omega t) &  \\
  -\sin (\omega t)& \cos (\omega t)&
\end{pmatrix}. \eqno (2.5)
$$

\end{ex}

\begin{lem}
For $m \ge 1$, let us consider a linear autonomous nonhomogeneous
ordinary differential equations of the first order
$$
{\frac{d}{d t}}((a_k)_{k \in \mathbb{N}})=
\Big(\sum_{n=0}^{2m} A_n\Big((\mathcal{F})\frac{\partial}{\partial
x}\Big)^{n}\Big) \times ((a_k)_{k \in \mathbb{N}})+ (f_k)_{k \in \mathbb{N}}
\eqno(2.6)
$$
with initial condition
$$
(a_k(0))_{k \in \mathbb{N}}=(C_k)_{k \in \mathbb{N}}, \eqno(2.7),
$$
where

$(i)$ $(C_k)_{k \in \mathbb{N}} \in {\bf R}^{\infty}$;

$(ii)$ $f=(f_k)_{k \in \mathbb{N}}$ is the sequence of continuous functions
of a parameter $t$ on $R$.

For each $k \ge 1$, we put
 $$\sigma_k=\sum_{n=0}^m (-1)^n A_{2n}(\frac{k \pi}{l})^{2n}, \eqno(2.8)$$
 $$ \omega_k=\sum_{n=0}^{m-1} (-1)^{n}A_{2n+1}(\frac{k\pi}{l})^{2n+1}.\eqno(2.9)$$

Then  the solution of (2.6)-(2.7) is given by
$$(a_k(t))_{k \in \mathbb{N}}=
e^{t(\sum_{n=0}^{2m} A_n\Big((\mathcal{F})\frac{\partial}{\partial
x}\Big)^{n})}\times (C_k)_{k \in \mathbb{N}}
+\int_{0}^te^{(\tau-t)(\sum_{n=0}^{2m}
A_n\Big((\mathcal{F})\frac{\partial}{\partial x}\Big)^{n})} \times
f( \tau)d \tau,\eqno(2.10)
$$
where $e^{t((\sum_{n=0}^{2m}
A_n\Big((\mathcal{F})\frac{\partial}{\partial x}\Big)^{n})}$
denotes  an exponent of the matrix $t(\sum_{n=0}^{2m}
A_n\Big((\mathcal{F})\frac{\partial}{\partial x}\Big)^{n})$ and it
exactly coincides with an infinite-dimensional  $(1,2,2,\dots)$
-cellular matrix $D(t)$ with cells $(D_k(t))_{k \in \mathbb{N}}$ for which
$D_0(t)=(e^{tA_0})$ and
$$D_k(t)= e^{\sigma_k t}\begin{pmatrix}
  \cos (\omega_k t) & \sin(\omega_k t)&  \\
  -\sin(\omega_k t) & \cos (\omega_k t)&
\end{pmatrix}, \eqno(2.11)$$
where for $k \ge 1$, $\sigma_k$ and $ \omega_k$ are defined by
(2.8)-(2.9), respectively.

\end{lem}

\begin{proof}We know that if we have  a linear autonomous inhomogeneous ordinary
differential equations of the first order
$$
{\frac{d}{d t}}((a_k)_{k \in \mathbb{N}})= E \times ((a_k)_{k
\in N})+ (f_k)_{k \in \mathbb{N}}     \eqno(2.12)
$$
with initial condition
$$
(a_k(0))_{k \in \mathbb{N}}=(C_k)_{k \in \mathbb{N}},  \eqno(2.13)
$$
where

$(i)$ $(C_k)_{k \in \mathbb{N}} \in {\bf R}^{\infty}$;

$(ii)$ $(f_k))_{k \in \mathbb{N}}$ is the sequence of continuous functions
of parameter $t$ on $R$;

$(iii)$  $E$ is an infinite dimensional $(1,2,2,\dots)$-cellular
matrix with cells $(E_k)_{k \in \mathbb{N}}$.

Then  the solution of (2.6)-(2.7) is given by (cf. \cite{Gant66}, \S 6, Section 1)
$$(a_k(t))_{k \in \mathbb{N}}=
e^{tE}\times (C_k)_{k \in \mathbb{N}} +\int_{0}^te^{(\tau-t)E} \times f(
\tau)d \tau,\eqno(2.14)
$$
where $e^{tE}$ and $e^{(\tau-t)E}$ denote exponents of matrices
$tE$ and $(\tau-t)E$, respectively.

Note that $t\sum_{n=0}^{2m}
A_n\Big((\mathcal{F})\frac{\partial}{\partial x}\Big)^{n}$ is an
infinite-dimensional $(1,2,2,\dots)$-cellular matrix with cells
$(tE_k)_{k \in \mathbb{N}}$ such that $tE_0=(tA_0)$ and
$$
tE_k=\begin{pmatrix}
 t \sigma_k& t \omega_k&  \\
 -t \omega_k& t \sigma_k&
\end{pmatrix}\eqno(2.15)
$$
for $k \ge 1$. Under notations (2.8)-(2.9), by using Example 2.4
we get that for $t \in R$, $e^{tE}$ exactly coincides with an
infinite-dimensional  $(1,2,2,\dots)$ -cellular matrix $D(t)$ with
cells $(D_k(t))_{k \in \mathbb{N}}$ for which $D_0(t)=(e^{tA_0})$ and
$$D_k(t)= e^{\sigma_k t}\begin{pmatrix}
  \cos (\omega_k t) & \sin(\omega_k t)&  \\
  -\sin(\omega_k t) & \cos (\omega_k t)&
\end{pmatrix}. \eqno(2.16).$$

Note that, for $ 0 \le \tau \le t$, the matrix $e^{(\tau-t)E}$
exactly coincides with an infinite-dimensional $(1,2,2,\dots)$
-cellular matrix $D(\tau-t)$.

\end{proof}

The following proposition is a simple consequence of Lemma 2.5.
\begin{cor}
 ~For $m \ge 1$, let us consider a linear partial
differential equation
$$
\frac{\partial}{\partial t}\Psi(t,x)=\sum_{n=0}^{2m}
A_n\frac{\partial^{n}}{\partial x^{n}}\Psi(t,x)~((t,x) \in
[0,+\infty[\times [-l,l[)\eqno(2.17)
$$
with initial condition
$$\Psi(0,x)=\frac{c_0}{2}+\sum_{k=1}^{\infty}c_k
\cos\Big(\frac{k \pi x}{l}\Big)+d_k \sin \Big(\frac{k \pi
x}{l}\Big) \in FD^{(0)}[-l,l[.\eqno(2.18)
$$
If $(\frac{c_0}{2}, c_1, d_1, c_2, d_2, \dots )$ is such a
sequence of real numbers that a series $\Psi(t,x)$ defined by
$$ \Psi(t,x)=\frac{e^{tA_0}c_0}{2}+\sum_{k=1}^{\infty} e^{\sigma_k t}\Big((c_k\cos (\omega_k t)+
$$
$$d_k\sin(\omega_k t))\cos(\frac{k\pi x}{l})+
(d_k\cos (\omega_k t)-c_k \sin(\omega_k t))\sin(\frac{k\pi
x}{l})\Big)\eqno(2.19)  $$
 belongs to the class  $FD^{(2m)}[-l,l[$ as a series  of a variable  $x$ for all $t \ge 0$, and
 is differentiable term by term  as a series  of a variable  $t$ for all $x \in [-l,l[$, then  $\Psi$ is a
 solution of $(2.17)$-$(2.18)$.
\end{cor}

\section{Solution of a linear  partial
differential equation of the higher order in two variables with
initial condition when coefficients are real-valued simple step
functions}

Let $0=t_0<\cdots <t_I=T$ and $-l=x_0<\cdots<x_J=l$. Suppose that
$$A_n(t,x)=\sum_{i=0}^{I-1}\sum_{j=0}^{J-1}A_n^{(i,j)}\times
\chi_{[t_i,t_{i+1}[\times [x_j,x_{j+1}[}(t,x).
$$
 ~For $m \ge 1$, let us consider a partial
differential equation
$$
\frac{\partial}{\partial t}\Psi(t,x)=\sum_{n=0}^{2m}
A_n(t,x)\frac{\partial^{n}}{\partial x^{n}}\Psi(t,x)~((t,x) \in
[0,T[\times [-l,l[)\eqno(3.1)
$$
with initial condition
$$\Psi(0,x)=\frac{c_0}{2}+\sum_{k=1}^{\infty}c_k
\cos\Big(\frac{k \pi x}{l}\Big)+d_k \sin \Big(\frac{k \pi
x}{l}\Big) \in FD^{(0)}[-l,l[. \eqno(3.2)
$$

\begin{defn}We say that $\Psi(t,x)$ is a weak solution of
(3.1)-(3.2) if the following conditions hold:

(i) ~$\Psi(t,x)$ satisfies (3.1) for each $(t,x) \in [0,T[\times
[-l,l[$ for which $t \neq t_i(0 \le i \le I)$ or $x \neq x_j (0
\le j \le J)$;

(ii)~$\Psi(t,x)$ satisfies (3.2);

(iii)~ for each fixed $x \in [-l,l[$, the function $\Psi(t,x)$ is continuous with respect to $t \in [0,T[$, and
for each $t \in
[0,T[$ the function $\Psi(t,x)$ is continuous with respect to $x$ on $[-l,l[$  except points  $\{ x_j: 0 \le j \le J-1 \}$.

\end{defn}

First, let fix $j$ and consider  a partial differential equation
$$
\frac{\partial}{\partial t}\Psi_{(0,j)}(t,x)=\sum_{n=0}^{2m}
A^{(0,j)}_n(t,x)\frac{\partial^{n}}{\partial
x^{n}}\Psi_{(0,j)}(t,x)~((t,x) \in [0,+\infty[\times
[-l,l[)\eqno(0.j)(PDE)
$$
with initial condition
$$\Psi_{(0,j)}(t_0,x)=\frac{c_0}{2}+\sum_{k=1}^{\infty}c_k
\cos\Big(\frac{k \pi x}{l}\Big)+d_k \sin \Big(\frac{k \pi
x}{l}\Big)= $$
$$\frac{c^{(0,j)}_0}{2}+\sum_{k=1}^{\infty}c^{(0,j)}_k
\cos\Big(\frac{k \pi x}{l}\Big)+d^{(0,j)}_k \sin \Big(\frac{k \pi
x}{l}\Big)   \in FD^{(0)}[-l,l[,\eqno(0,j)(IC)
$$

By Corollary 2.6, under some restrictions on $(\frac{c_0}{2}, c_1,
d_1, c_2, d_2, \dots )$, a series $\Psi_{(0,j)}(t,x)$ defined by
$$ \Psi_{(0,j)}(t,x)=\frac{e^{tA^{(0,j)}_0}c^{(0,j)}_0}{2}+\sum_{k=1}^{\infty} e^{\sigma^{(0,j)}_k t}\Big((c^{(0,j)}_k\cos (\omega^{(0,j)}_k t)+
$$
$$d^{(0,j)}_k\sin(\omega^{(0,j)}_k t))\cos(\frac{k\pi x}{l})+
(d^{(0,j)}_k\cos (\omega^{(0,j)}_k t)-c^{(0,j)}_k
\sin(\omega^{(0,j)}_k t))\sin(\frac{k\pi x}{l})\Big)\eqno(3.3)
$$ is a
 solution of $(0.j)(PDE)$-$(0,j)(IC)$,

Now let consider a partial differential equation
$$
\frac{\partial}{\partial t}\Psi_{(1,j)}(t,x)=\sum_{n=0}^{2m}
A^{(1,j)}_n(t,x)\frac{\partial^{n}}{\partial
x^{n}}\Psi_{(1,j)}(t,x)~((t,x) \in [0,+\infty[\times
[-l,l[)\eqno(1.j)(PDE)
$$
with initial condition
$$\Psi_{(1,j)}(t_1,x)=\Psi_{(0,j)}(t_1,x) \eqno(1,j)(IC)$$

We wiil try to present  the solution of the  $(1,j)(PDE)$ by the following form

$$
 \Psi_{(1,j)}(t,x)=\frac{e^{tA^{(1,j)}_0}c^{(1,j)}_0}{2}+\sum_{k=1}^{\infty} e^{\sigma^{(1,j)}_k t}\Big((c^{(1,j)}_k\cos (\omega^{(1,j)}_k t)+
$$
$$
d^{(1,j)}_k\sin(\omega^{(1,j)}_k t))\cos(\frac{k\pi x}{l})+
(d^{(1,j)}_k\cos (\omega^{(1,j)}_k t)-c^{(1,j)}_k
\sin(\omega^{(1,j)}_k t))\sin(\frac{k \pi x}{l})\Big)\eqno(3.4)
$$

In order to get validity of the condition  $(1,j)(IC)$, we  consider  the following  infinite system of equations:

 $$
 \frac{e^{t_1A^{(1,j)}_0}c^{(1,j)}_0}{2}=\frac{e^{t_1A^{(0,j)}_0}c^{(0,j)}_0}{2},\eqno(3.5)
$$

$$
 e^{\sigma^{(1,j)}_k t_1}(c^{(1,j)}_k\cos (\omega^{(1,j)}_k t_1)+d^{(1,j)}_k\sin(\omega^{(1,j)}_k t_1))=$$

$$
e^{\sigma^{(0,j)}_k t_1}(c^{(0,j)}_k\cos (\omega^{(0,j)}_k t_1)+d^{(0,j)}_k\sin(\omega^{(0,j)}_k t_1))(k \in \mathbb{N}),\eqno(3.6)
$$

$$
 e^{\sigma^{(1,j)}_k t_1}(d^{(1,j)}_k\cos (\omega^{(1,j)}_k t_1)-c^{(1,j)}_k
\sin(\omega^{(1,j)}_k t_1))= e^{\sigma^{(0,j)}_k t_1}(d^{(0,j)}_k\cos (\omega^{(0,j)}_k t_1)-c^{(0,j)}_k
\sin(\omega^{(0,j)}_k t_1)) (k \in \mathbb{N}).\eqno(3.7)
$$

We have
$$
 c^{(1,j)}_0=e^{t_1(A^{(0,j)}_0-A^{(1,j)}_0)}c^{(0,j)}_0.\eqno(3.8)
$$
For $k \in \mathbb{N}$ we can rewrite  equations (3.6)-(3.7) as follows:

$$
 c^{(1,j)}_k\cos (\omega^{(1,j)}_k t_1)+d^{(1,j)}_k\sin(\omega^{(1,j)}_k t_1))=
e^{(\sigma^{(0,j)}_k -\sigma^{(1,j)}_k )t_1}(c^{(0,j)}_k\cos (\omega^{(0,j)}_k t_1)+d^{(0,j)}_k\sin(\omega^{(0,j)}_k t_1)), \eqno(3.9)
$$

$$
 -c^{(1,j)}_k \sin(\omega^{(1,j)}_k t_1)+d^{(1,j)}_k\cos (\omega^{(1,j)}_k t_1) = e^{(\sigma^{(0,j)}_k - \sigma^{(1,j)}_k)  t_1}(d^{(0,j)}_k\cos (\omega^{(0,j)}_k t_1)-c^{(0,j)}_k\sin(\omega^{(0,j)}_k t_1)). \eqno(3.10)
$$
Setting
$$\mathbb{A}=e^{(\sigma^{(0,j)}_k -\sigma^{(1,j)}_k )t_1}(c^{(0,j)}_k\cos (\omega^{(0,j)}_k t_1)+d^{(0,j)}_k\sin(\omega^{(0,j)}_k t_1))\eqno(3.11)
$$
and
$$
\mathbb{B}=e^{(\sigma^{(0,j)}_k - \sigma^{(1,j)}_k)  t_1}(d^{(0,j)}_k\cos (\omega^{(0,j)}_k t_1)-c^{(0,j)}_k
\sin(\omega^{(0,j)}_k t_1)),\eqno(3.12)
$$
for $k \in \mathbb{N}$  we obtain
$$
 c^{(1,j)}_k\cos (\omega^{(1,j)}_k t_1)+d^{(1,j)}_k\sin(\omega^{(1,j)}_k t_1)=\mathbb{A} \eqno(3.13)
 $$
and
$$
 -c^{(1,j)}_k
\sin(\omega^{(1,j)}_k t_1)+d^{(1,j)}_k\cos (\omega^{(1,j)}_k t_1) = \mathbb{B}.\eqno(3.14)
$$
It is obvious that the system of equations (3.13)-(3.14) has the unique solution which can be done as follows:
$$
c^{(1,j)}_k=\mathbb{A}\cos (\omega^{(1,j)}_k t_1)-\mathbb{B}\sin(\omega^{(1,j)}_k t_1)\eqno(3.15)
$$
and
$$
d^{(1,j)}_k=\mathbb{B}\cos (\omega^{(1,j)}_k t_1)+\mathbb{A}\sin(\omega^{(1,j)}_k t_1)\eqno(3.16)
$$
for $k \in \mathbb{N}$.

By Corollary 2.6, under some restrictions on
$(\frac{c^{(1,j)}_0}{2}, c^{(1,j)}_1, d^{(1,j)}_1, c^{(1,j)}_2,
d^{(1,j)}_2, \dots )$, the series $\Psi_{(1,j)}(t,x)$ defined by (3.4) is the
 solution of $(1.j)(PDE)$-$(1,j)(IC)$,

It is obvious that under nice restrictions on coefficients participated in (3.1) and (3.2), we  can continue our procedure step by step. Correspondingly we can construct a sequence
$(\Psi_{(s,j)})_{0 \le s \le I-1, 1 \le j \le J-1}$ such that
$\Psi_{(s,j)}$ satisfies a linear partial differential equation
$$
\frac{\partial}{\partial t}\Psi_{(s,j)}(t,x)=\sum_{n=0}^{2m}
A^{(k,j)}_n(t,x)\frac{\partial^{n}}{\partial
x^{n}}\Psi_{(s,j)}(t,x)~((t,x) \in [0,+\infty[\times
[-l,l[)\eqno(s,j)(PDE)
$$
with initial condition
$$\Psi_{(s,j)}(t_s,x)=\Psi_{(s-1,j)}(t_s,x)=
\frac{c^{(s,j)}_0}{2}+\sum_{k=1}^{\infty}c^{(s,j)}_k
\cos\Big(\frac{k \pi x}{l}\Big)+d^{(s,j)}_k \sin \Big(\frac{k \pi
x}{l}\Big)\eqno(s,j)(IC).
$$

Now it is obvious to observe that we have proved the validity of
the following assertion.

\begin{thm}If for coefficients $(\frac{c^{(i,j)}_0}{2}, c^{(i,j)}_1, d^{(i,j)}_1,
c^{(i,j)}_2, d^{(i,j)}_2, \dots )(1 \le i \le I,1 \le j \le J) $
functions $\Psi_{(i,j)}(t,x)$ satisfy conditions of Corollary 2.6,
then a function $\Psi(t,x):[0,T[\times [-l,l[\to R $ defined by
$$\sum_{i=0}^{I-1}\sum_{j=0}^{J-1}\Psi_{(i,j)}(x,t)\times
\chi_{[t_i,t_{i+1}[\times [x_j,x_{j+1}[}(t,x)\eqno(3.17)$$ is a
weak solution of (3.1)-(3.2).
\end{thm}

\medskip

\begin{ex} Let consider a linear partial differential equation of the $22$
order in two variables
$$
\frac{\partial}{\partial t}\Psi(t,x)= A(t,x)\times
\frac{\partial^{2}}{\partial x^{2}}\Psi(t,x)+B(t,x)\times
\frac{\partial^{22}}{\partial x^{22}}\Psi(t,x)~((t,x) \in
[0,2\pi[\times [0,\pi[)\eqno(3.18)
$$
with initial condition
$$\Psi(0,x)=\frac{0.015}{2}+5\sin (x), \eqno(3.19)
$$
where
$$ A(t,x)=\chi_{[0,\pi[\times [0,\pi[}(t,x)+1.55 \times
\chi_{[\pi,2\pi[\times [0,\pi[}(t,x)$$

and

$$ B(t,x)=2\times \chi_{[0,\pi[\times [0,\pi[}(t,x)-2 \times
\chi_{[\pi,2\pi[\times [0,\pi[}(t,x).$$

\medskip

The programm in MathLab for a solution of $(3.18)-(3.19)$, has the
following form:{\bf

\medskip

$A1=[0,1,0,0,0,0,0,0,0,0,0,0,0,0,0,0,0,0,0,0,0,2];$

$A2=[0,1.55,0,0,0,0,0,0,0,0,0,0,0,0,0,0,0,0,0,0,0,-2];$

$C1=[0,0,0,0,0,0,0,0,0,0, 0,0,0,0,0, 0,0,0,0,0];$

$D1=[5,0,0,0,0,0,0,0,0,0, 0,0,0,0,0, 0,0,0,0,0];$

$A10=0; A20=0; C10=0.015;$

for $k=1:20$

$S1(k)=A10;   S2(k)=A20;$

for $n=1:10$

$S1(k)=S1(k)+(-1)^(n)*A1(2*n)*k^(2*n);$

$S2(k)=S2(k)+(-1)^(n)*A2(2*n)*k^(2*n);$

end

end

for $k=1:20$

$O1(k)=0;$

$O2(k)=0;$

end

for $k=1:20$

for $n=1:10$

$O1(k)=O1(k) +(-1)^n*A1(2*n+1)*k^(2*n+1);$

$O2(k)=O2(k) +(-1)^n*A2(2*n+1)*k^(2*n+1);$

end

end

$[T1,X1]=\mbox{meshgrid}(0:( pi/10):pi, 0 : ( pi/10):pi);$

$Z1=0.5* C10*\exp(T1.*A10);$

for  $k=1:20$

$Z1=Z1+ C1(k)*\exp(T1*S1(k)).*\cos(X1.*k).* \cos(T1*O1(k))+D1(1)*\exp(T1*S1(k)).*\cos(X1.*k).* \sin(T1*O1(k))+ $

$D1(k)*\exp(T1*S1(k)).* \sin(X1.*k).* \cos(T1*O1(k))-   C1(k)*\exp(T1*S1(k)).* \sin(X1.*k).* \sin(T1*O1(k));$

end

$C20= \exp (pi* (A10-A20))* C10;$

for $k=1:20$

$A(k)= \exp( (S1(k)- S2(k))*pi) *(C1(k) *\cos (O1(k)*pi)+D1(k)* \sin(O1(k)*pi));$

$B(k)= \exp( (S1(k)- S2(k))*pi)* (D1(k)* \cos (O1(k)*pi)-C1(k)* \sin(O1(k)*pi));$

end

for  $k=1:20$

$C2(k) =A(k)*\cos (O2(k)*pi)- B(k)*\sin(O2(k)*pi);$

$D2(k) =B(k)*\cos (O2(k)*pi)+ A(k)*\sin(O2(k)*pi);$

end

$[T2,X2]=\mbox{meshgrid} (pi:( pi/10): (2*pi), 0 : ( pi/10):pi);$

$Z2=0.5* C20* exp((T2)*A20);$

for  $k=1:20$

$Z2=Z2+ C2(k)*\exp(T2*S2(k)).*\cos(X2.*k).* \cos(T2*O2(k))+D2(1)*\exp(T2*S1(k)).*\cos(X2.*k).* \sin(T2*O2(k))+ $

$D2(k)*\exp(T2*S2(k)).* \sin(X2.*k).* \cos(T2*O2(k))-   C2(k)*\exp(T2*S2(k)).* \sin(X2.*k).* \sin(T2*O2(k));$

end

$\mbox{surf}(T1,X1,Z1)$

hold on

$\mbox{surf}(T2,X2,Z2)$

hold off}

\begin{figure}[h]
\center{\includegraphics[width=1\linewidth]{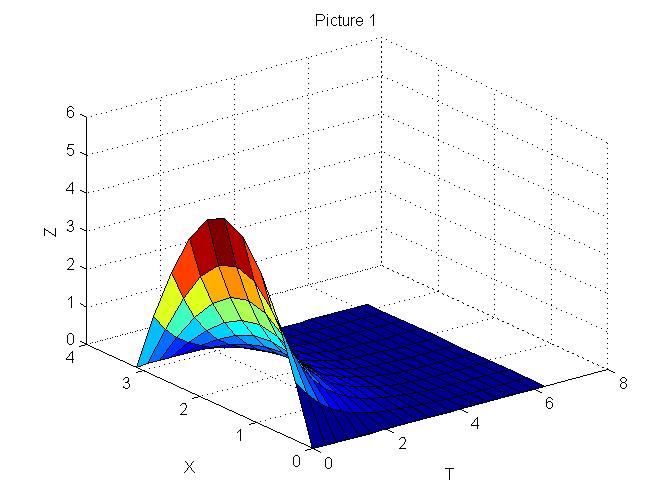}}
\caption{Graphic of the solution of the LPDE-(3.18) with IC-(3.19)).} \label{ris:image}
\end{figure}

\end{ex}

\begin{ex}Let consider a linear partial differential equation of the $22$
order in two variables
$$
\frac{\partial}{\partial t}\Psi(t,x)=A(t,x)\Psi(t,x)+ B(t,x)\times
\frac{\partial^{2}}{\partial x^{2}}\Psi(t,x)+C(t,x)\times
\frac{\partial^{3}}{\partial x^{3}}\Psi(t,x)+
$$
\begin{figure}[h]
\center{\includegraphics[width=1\linewidth]{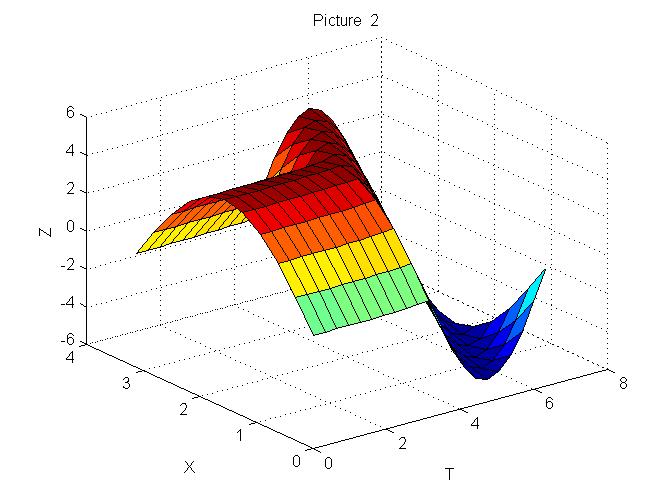}}
\caption{Graphic of the solution of the LPDE-(3.20) with IC-(3.21)).} \label{ris:image}
\end{figure}

$$D(t,x)\times
\frac{\partial^{22}}{\partial x^{22}}\Psi(t,x)~((t,x) \in
[0,2\pi[\times [0,\pi[)\eqno(3.20)
$$
with initial condition
$$\Psi(0,x)=\frac{0.015}{2}+5\sin (x), \eqno(3.21)
$$
where
$$ A(t,x)=1\chi_{[0,\pi[\times [0,\pi[}(t,x)+0 \times
\chi_{[\pi,2\pi[\times [0,\pi[}(t,x),$$
$$ B(t,x)=\chi_{[0,\pi[\times [0,\pi[}(t,x)+0 \times
\chi_{[\pi,2\pi[\times [0,\pi[}(t,x),$$
$$ C(t,x)=0\times \chi_{[0,\pi[\times [0,\pi[}(t,x)+1 \times
\chi_{[\pi,2\pi[\times [0,\pi[}(t,x)$$
and
$$ D(t,x)=2\times \chi_{[0,\pi[\times [0,\pi[}(t,x)+2 \times
\chi_{[\pi,2\pi[\times [0,\pi[}(t,x).$$

\medskip

The graphical solution of (3.20)-(3.21) can be obtained by MathLab programm used in Example 3.3 for the following data:

$A1=[0,1,0,0,0,0,0,0,0,0,0,0,0,0,0,0,0,0,0,0,0,2];$

$A2=[0,0,1,0,0,0,0,0,0,0,0,0,0,0,0,0,0,0,0,0,0,2];$

$C1=[0,0,0,0,0,0,0,0,0,0, 0,0,0,0,0, 0,0,0,0,0];$

$D1=[5,0,0,0,0,0,0,0,0,0, 0,0,0,0,0, 0,0,0,0,0];$

$A10=1; A20=0; C10=0.015;$

\end{ex}

\begin{ex}Let consider a linear partial differential equation of the $21$
order in two variables with constant coefficients
$$
\frac{\partial}{\partial t}\Psi(t,x)=A(t,x)\Psi(t,x)+
\frac{\partial^{2}}{\partial x^{2}}\Psi(t,x)+2
\frac{\partial^{21}}{\partial x^{21}}\Psi(t,x)~((t,x) \in
[0,2\pi[\times [0,\pi[)\eqno(3.22)
$$
\begin{figure}[h]
\center{\includegraphics[width=1\linewidth]{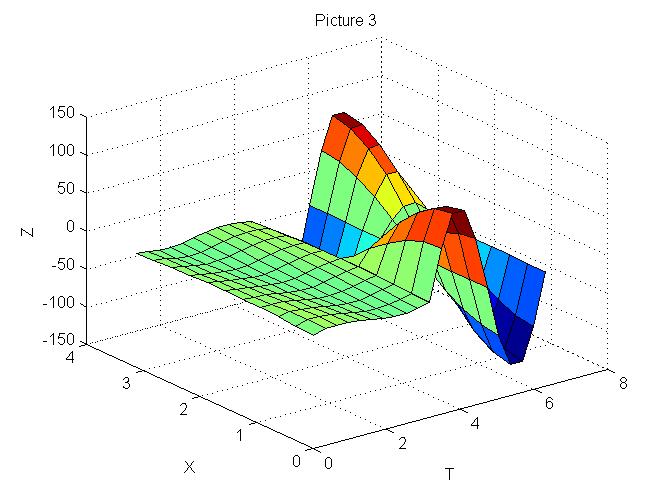}}
\caption{Graphic of the solution of the LPDE-(3.22) with IC-(3.23)).} \label{ris:image}
\end{figure}
with initial condition
$$\Psi(0,x)=\frac{0.15}{2}+5\sin (x), \eqno(3.23)
$$
where
$$ A(t,x)=1\chi_{[0,\pi[\times [0,\pi[}(t,x)+0 \times
\chi_{[\pi,2\pi[\times [0,\pi[}(t,x).$$

Since
$$ 1=1 \times\chi_{[0,\pi[\times [0,\pi[}(t,x)+1 \times
\chi_{[\pi,2\pi[\times [0,\pi[}(t,x)$$
and
$$ 2=2\times \chi_{[0,\pi[\times [0,\pi[}(t,x)+2 \times
\chi_{[\pi,2\pi[\times [0,\pi[}(t,x),$$
the graphical solution of (3.22)-(3.23) can be obtained by MathLab programm used in Example 3.3 for the following data:

$A1=[0,1,0,0,0,0,0,0,0,0,0,0,0,0,0,0,0,0,0,0,2,0];$

$A2=[0,1,0,0,0,0,0,0,0,0,0,0,0,0,0,0,0,0,0,0,2,0];$

$C1=[0,0,0,0,0,0,0,0,0,0, 0,0,0,0,0, 0,0,0,0,0];$

$D1=[5,0,0,0,0,0,0,0,0,0, 0,0,0,0,0, 0,0,0,0,0];$

$A10=1; A20=0; C10=0.15;$

\end{ex}

\begin{ex}
Let consider a linear partial differential equation of the $21$
order in two variables
$$
\frac{\partial}{\partial t}\Psi(t,x)=A(t,x)\Psi(t,x)+ B(t,x)\times
\frac{\partial^{2}}{\partial x^{2}}\Psi(t,x)+100\frac{\partial^{3}}{\partial x^{3}}\Psi(t,x)
2\frac{\partial^{21}}{\partial x^{21}}\Psi(t,x)~((t,x) \in
[0,2\pi[\times [0,\pi[)\eqno(3.24)
$$
\begin{figure}[h]
\center{\includegraphics[width=1\linewidth]{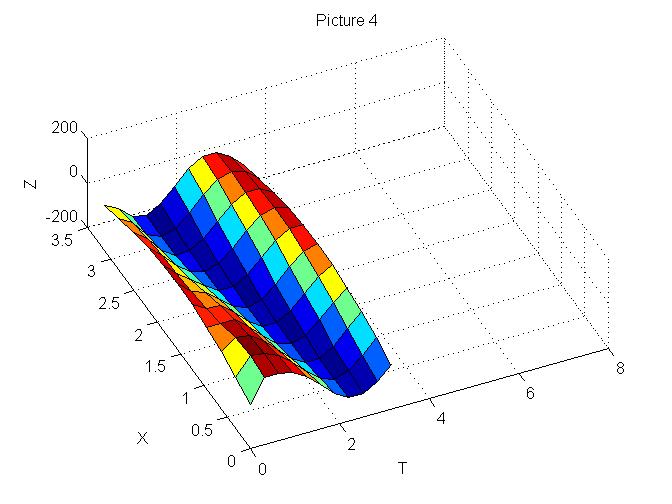}}
\caption{Graphic of the solution of the LPDE-(3.24) with IC-(3.25)).} \label{ris:image}
\end{figure}

with initial condition
$$\Psi(0,x)=\frac{0.015}{2}+100\sin (x), \eqno(3.25)
$$
where
$$ A(t,x)=1\chi_{[0,\pi[\times [0,\pi[}(t,x)+0
\chi_{[\pi,2\pi[\times [0,\pi[}(t,x)$$
and
$$ B(t,x)=\chi_{[0,\pi[\times [0,\pi[}(t,x)-
\chi_{[\pi,2\pi[\times [0,\pi[}(t,x).$$

\medskip

The graphical solution of (3.24)-(3.25) can be obtained by MathLab programm used in Example 3.3 for the following data:

$A1=[0,1,100,0,0,0,0,0,0,0,0,0,0,0,0,0,0,0,0,0,2,0];$

$A2=[0,-1,100,0,0,0,0,0,0,0,0,0,0,0,0,0,0,0,0,0,2,0];$

$C1=[0,0,0,0,0,0,0,0,0,0, 0,0,0,0,0, 0,0,0,0,0];$

$D1=[100,0,0,0,0,0,0,0,0,0, 0,0,0,0,0, 0,0,0,0,0];$

$A10=1; A20=0; C10=0.15;$

We see that we have no graphic on the region $[\pi,2\pi[\times [0,\pi[$ which hints us that  coefficients of the
LPDE (3.24)-(3.25)  on that region do not satisfy conditions of Theorem 3.2.
\end{ex}

\begin{ex}Let consider a linear partial differential equation of the $21$
order  in two variables
$$
\frac{\partial}{\partial t}\Psi(t,x)=A(x,t)\Psi(t,x)+
50\frac{\partial^5}{\partial x^5}\Psi(t,x)+ \frac{\partial^6}{\partial x^6}\Psi(t,x) + 2
\frac{\partial^{21}}{\partial x^{21}}\Psi(t,x)~((t,x) \in
[0,2\pi[\times [0,\pi[)\eqno(3.26)
$$
\begin{figure}[h]
\center{\includegraphics[width=1\linewidth]{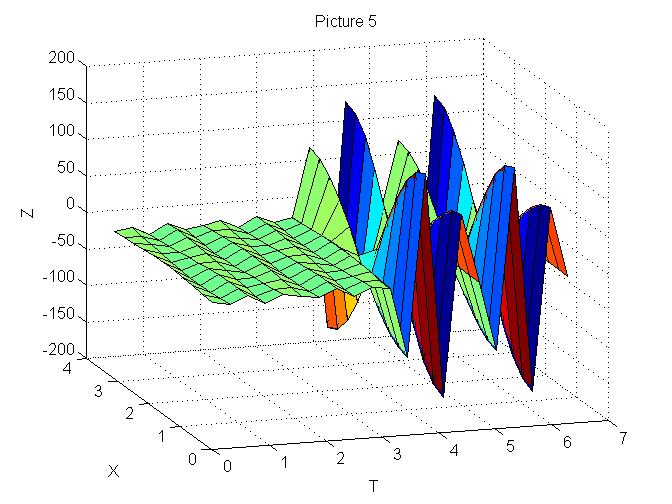}}
\caption{Graphic of the solution of the LPDE-(3.26) with IC-(3.27)).} \label{ris:image}
\end{figure}
with initial condition
$$\Psi(0,x)=\frac{0.15}{2}+5\sin (x), \eqno(3.27).
$$
where
$$ A(t,x)=\chi_{[0,\pi[\times [0,\pi[}(t,x)+0\times
\chi_{[\pi,2\pi[\times [0,\pi[}(t,x).$$

The graphical solution of (3.26)-(3.27) can be obtained by MathLab programm used in Example 3.3 for the following data:

$A1=[0,0,0,0,50,1,0,0,0,0,0,0,0,0,0,0,0,0,0,0,2,0];$

$A2=[0,0,0,0,50,1,0,0,0,0,0,0,0,0,0,0,0,0,0,0,2,0];$

$C1=[0,0,0,0,0,0,0,0,0,0, 0,0,0,0,0, 0,0,0,0,0];$

$D1=[7,0,0,0,0,0,0,0,0,0, 0,0,0,0,0, 0,0,0,0,0];$

$A10=1; A20=0; C10=0.15;$

\end{ex}

\begin{figure}[h]
\center{\includegraphics[width=1\linewidth]{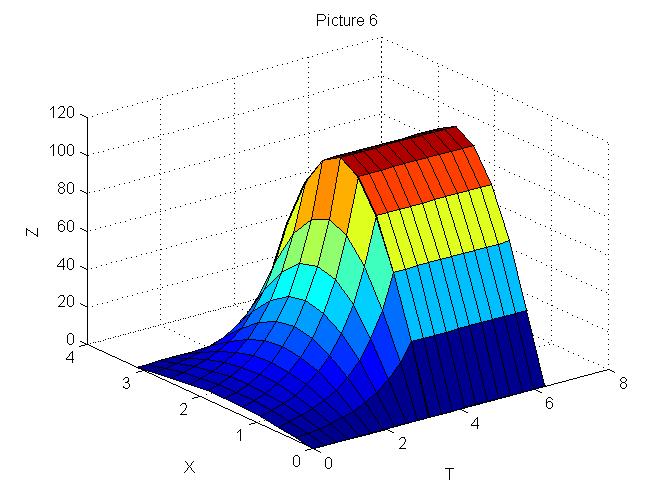}}
\caption{Graphic of the solution of the LPDE-(3.28) with IC-(3.29)).} \label{ris:image}
\end{figure}
\begin{ex}Let consider a linear partial differential equation of the $22$
order  in two variables
$$
\frac{\partial}{\partial t}\Psi(t,x)=A(x,t)\Psi(t,x)-
\frac{\partial}{\partial x}\Psi(t,x)+2
\frac{\partial^{22}}{\partial x^{22}}\Psi(t,x)~((t,x) \in
[0,2\pi[\times [0,\pi[)\eqno(3.28)
$$

with initial condition
$$\Psi(0,x)=\frac{0.015}{2}+5\sin (x), \eqno(3.29).
$$
where
$$ A(t,x)=\chi_{[0,\pi[\times [0,\pi[}(t,x)+0\times
\chi_{[\pi,2\pi[\times [0,\pi[}(t,x).$$

The graphical solution of (3.28)-(3.29) can be obtained by MathLab programm used in Example 3.3 for the following data:

$A1=[-1,0,0,0,0,0,0,0,0,0,0,0,0,0,0,0,0,0,0,0,0,2];$

$A2=[-1,0,0,0,0,0,0,0,0,0,0,0,0,0,0,0,0,0,0,0,0,2];$

$C1=[0,0,0,0,0,0,0,0,0,0, 0,0,0,0,0, 0,0,0,0,0];$

$D1=[5,0,0,0,0,0,0,0,0,0, 0,0,0,0,0, 0,0,0,0,0];$

$A10=1; A20=0; C10=0.015;$

\end{ex}


\begin{rem}
 Notice that for each natural number $M>1$, one
can easily modify the MathLab program described in Example 3.3 for obtaining the
graphical solution of the linear partial differential equation
(3.1)-(3.2) whose coefficients $(A_n(t,x))_{0 \le n \le 2M}$ are
real-valued simple step functions on $[0,T[ \times [-l,l[$ and $f$
is a trigonometric polynomial on $[-l,l[$.
\end{rem}

\begin{rem} Since each  constant $c$  same times is a step
function  we can use MathLab program described in Example 3.3 for a solution of the linear partial
differential equation (2.17)-(2.18) with constant coefficients.

For example, if we consider LPDE
$$
\frac{\partial}{\partial t}\Psi(t,x)=-\frac{\partial}{\partial
x}\Psi(t,x) + \frac{\partial^{3}}{\partial
x^{3}}\Psi(t,x)-\frac{\partial^{5}}{\partial
x^{5}}\Psi(t,x)+ \frac{\partial^{7}}{\partial
x^{7}}\Psi(t,x)+2 \frac{\partial^{22}}{\partial
x^{22}}\Psi(t,x)~((t,x) \in [0,2\pi[\times [0,\pi[)\eqno(3.30)
$$
\begin{figure}[h]
\center{\includegraphics[width=1\linewidth]{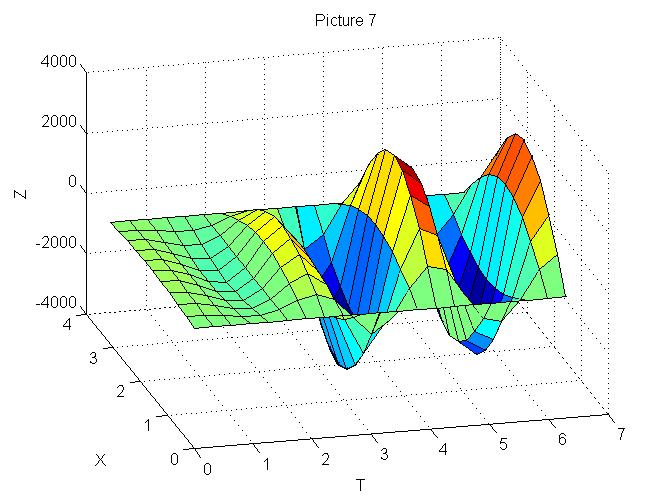}}
\caption{Graphic of the solution of the LPDE-(3.30) with IC-(3.1)).} \label{ris:image}
\end{figure}

with initial condition
$$\Psi(0,x)=\frac{0.015}{2}+150 \sin (x), \eqno(3.31)
$$
then for a solution  (3.30)-(3.31),in MathLab programm  described in Example 3.3  we must enter  the following data:

$A1=[-1,0,1,0,-1,0,1,0,0,0,0,0,0,0,0,0,0,0,0,0,0,2];$

$A2=[-1,0,1,0,-1,0,1,0,0,0,0,0,0,0,0,0,0,0,0,0,0,2];$

$C1=[0,0,0,0,0,0,0,0,0,0, 0,0,0,0,0, 0,0,0,0,0];$

$D1=[150,0,0,0,0,0,0,0,0,0, 0,0,0,0,0, 0,0,0,0,0];$

$A10=1; A20=0; C10=0.015;$

\end{rem}

\begin{rem} The approach used for a solution of
(3.1)-(3.2) can be used in such a case when coefficients
$(A_n(t,x))_{0\le n \le 2M}$ are rather smooth continuous functions on
$[0,T[\times [-l,l[$. If we will approximate $(A_n(t,x))_{0\le n
\le 2M}$ by real-valued simple step functions, then it is natural
to wait that under some "nice restrictions" on $(A_n(t,x))_{0\le n
\le 2M}$ the solutions obtained by Theorem 3.2, will give us a
"good approximation" of the solution of the required linear
partial differential equation of the higher  order in two variables  with corresponding initial
conditions.
\end{rem}



\end{document}